\documentclass[11pt]{amsart}
\usepackage{amssymb, mathrsfs, amsfonts, amsmath}
\usepackage{graphicx,color}
\usepackage{epsfig}
\usepackage[all]{xy}
\usepackage{subcaption}
\usepackage{verbatim}
\usepackage{hyperref}
\usepackage{enumitem}
\usepackage{tikz}
\newcommand{\F}{\mathbb F}

\newcommand{\R}{\mathbb R}

\newtheorem{theorem}{Theorem}[section]

\newtheorem{proposition}[theorem]{Proposition}
\newtheorem{definition}[theorem]{Definition}

\newtheorem{remark}[theorem]{Remark}
\newtheorem{afirm}[theorem]{Affirmation}
\newtheorem{conjecture}[theorem]{Conjecture}

\newtheorem{innercustomthm}{Theorem}
\newenvironment{customthm}[1]
  {\renewcommand\theinnercustomthm{#1}\innercustomthm}
  {\endinnercustomthm}

\begin{document}
\title[Universality theorem]{Universality theorem for LNE H\"older triangles}

\author[L. Birbrair]{Lev Birbrair}
\address{Departamento de Matem\'atica, Universidade Federal do Cear\'a (UFC), Campus do Pici, Bloco 914, Cep.~60455-760, Fortaleza-Ce, Brasil}
\address{Faculty of Mathematics and Computer Science, Jagiellonian University, Prof.~Stanis\l awa \L ojasiewicza 6, 30-348, Krak\'ow, Poland}
\email{lev.birbrair@gmail.com}

\author[M. Denkowski]{Maciej Denkowski}
\address{Faculty of Mathematics and Computer Science, Jagiellonian University, Prof.~Stanis\l awa \L ojasiewicza 6, 30-348, Krak\'ow, Poland}
\email{maciej.denkowski@uj.edu.pl}

\author[D. L. Medeiros]{Davi Lopes Medeiros}
\address{Departamento de Matem\'atica, Universidade Federal do Cear\'a (UFC), Campus do Pici, Bloco 914, Cep.~60455-760, Fortaleza-Ce, Brasil}
\email{profdavilopes@gmail.com}

\author[J. E. Sampaio]{Jos\'e Edson Sampaio}
\address{Departamento de Matem\'atica, Universidade Federal do Cear\'a (UFC), Campus do Pici, Bloco 914, Cep.~60455-760, Fortaleza-Ce, Brasil}\email{edsonsampaio@mat.ufc.br}

\thanks{The last named author was partially supported by CNPq-Brazil grant 310438/2021-7 and supported by the Serrapilheira Institute (grant number Serra -- R-2110-39576).
}
\begin{abstract}
    We compare ambient and outer Lipschitz geometry of Lipschitz normally embedded H\"older triangles in $\R^4$. In contrast to the case of $\R^3$ there are infinitely many equivalence classes. The equivalence classes are related to the so-called microknots.
\end{abstract}
\maketitle

\tableofcontents
\section{Introduction}
The paper is devoted to the metric properties of real semialgebraic (or definable in a polynomially bounded o-minimal structure) surfaces. We study the difference between the \emph{outer} and \emph{ambient} bi-Lipschitz equivalence of semialgebraic surface germs at the origin in $\R^4$. Two surface germs are outer bi-Lipschitz equivalent if they are bi-Lipschitz equivalent as abstract metric spaces with the outer metric $d(x,y)=\|x-y\|$.
 Ambient bi-Lipschitz equivalence means that there exists a germ of a bi-Lipschitz homeomorphism of the ambient space mapping one of them to the other one. 
Birbrair and Gabrielov in \cite{BG:2019} showed that ambient bi-Lipschitz equivalence is different from outer bi-Lipschitz equivalence even when there are no topological obstructions. More precisely, they showed that {\it there are two germs at the origin of semialgebraic surfaces $X_1$ and $X_2$ in $\mathbb{R}^3$ that are outer bi-Lipschitz equivalent and ambient topologically equivalent, but are not ambient bi-Lipschitz equivalent} (see \cite[Theorem 2.1]{BG:2019}). 
Indeed, these two equivalence relations are reasonably different. This is evidenced by the Universality Theorem proved by Birbrair, Brandenbursky and Gabrielov in \cite[Theorem 3.1]{withMisha} (see also \cite{handbook}):
\begin{theorem}[Universality Theorem] Let $K \subset \mathbb{S}^3$ be a knot. Then one can associate
to $K$ a semialgebraic one-bridge surface germ $(X_K , 0)$ in $\R^4$ so that the following holds:
\begin{itemize}                                                                                    \item [1.] The link at the origin of each germ $X_K$ is a trivial knot;                                                                                               \item [2.] All germs $X_K$ are outer bi-Lipschitz equivalent;                                                                                               \item [3.] Two germs $X_{K_1}$ and $X_{K_2}$ are ambient semialgebraic bi-Lipschitz equivalent only if the knots $K_1$ and $K_2$ are isotopic.                                                                                                 \end{itemize}
\end{theorem}

Here the {\it link} of a germ $(X,0)$ in $\mathbb{R}^n$ (Definition \ref{Def: Link at the origin}) is the intersection $X\cap\mathbb{S}_\varepsilon^{n-1}$, for $\varepsilon>0$ small enough --- by the Local Conical Structure Theorem, it is uniquely determined up to a bi-Lipschitz homeomorphism. In the Theorem cited above, although the links at the origin of all surface germs $X_K$ are trivial knots, the map $K\mapsto X_K$ from the set of all isotopy classes of knots in $\mathbb{S}^3$ to the set of ambient bi-Lipschitz equivalence classes of surface germs in $\mathbb{R}^4$ is injective.
In particular, this theorem implies that the ambient bi-Lipschitz classification in this case ``contains all of Knot Theory''.

In this article, we prove a Universality Theorem where $X_K$ has the simplest topology. Indeed, we show that we can take $X_K$ to be an LNE H\"older triangle in $\R^4$, which is the most simple object of the metric theory of singularities. Recall that a subset $X\subset \R^n$ is called {\bf Lipschitz normally embedded} ({\bf LNE}, for short) if its inner metric is equivalent to its outer metric. This condition is not satisfied for the aforementioned example of Birbrair and Gabrielov.

An LNE $\beta$-H\"older triangle $T\subset \R^n$ (resp. LNE $\beta$-horn $H\subset \R^n$) is called {\bf ambient Lipschitz trivial}  if it is ambient bi-Lipschitz equivalent to a standard $\beta$-H\"older triangle (resp. LNE $\beta$-horn) in a two (resp. three) dimensional subspace in $\R^n$. Recently, Birbrair and Medeiros \cite{withDavi} have shown that any LNE $\beta$-H\"older triangle and LNE $\beta$-horn in $\R^3$ are ambient Lipschitz trivial. The results of \cite{extention} prove that for $n>4$ any LNE $\beta$-H\"older triangle and LNE $\beta$-horn are also ambient Lipschitz trivial. 
Here, the sets are assumed to be definable in a polynomially bounded o-minimal structure over $\mathbb{R}$ with the field of exponents $\mathbb{F}$.
Thus, we can summarize the above-mentioned results from \cite{withDavi} and \cite{extention} as the following:

\begin{theorem}\label{thm:trivial_holder_triangle}
	For any $\beta \in \mathbb{F}_{\ge 1}$, $n \in \mathbb{N}_{\ge 3}$ with $n\ne 4$, every LNE $\beta$-H\"older triangle $T \subset \R^n$ and every LNE $\beta$-horn $H \subset \R^n$ are ambient Lipschitz trivial.
\end{theorem}

In this paper, we investigate the case $n=4$ and we obtain the following surprising result (see the notion of a microknot in Definition \ref{Def:microknot}):
\begin{customthm}{\ref*{main}} 
For any knot $K$ and $\beta \in \mathbb{F}_{> 1}$, there exists an LNE H\"older triangle $T_{\beta, K} \subset \R^4$, containing a $\beta$-microknot, isotopic to $K$. Moreover, two triangles $T_{\beta, K_1}$ and $T_{\beta, K_2}$ are ambient bi-Lipschitz equivalent only if $K_1$ and $K_2$ are isotopic.
\end{customthm}

 The above result gives also a negative answer to the following conjecture stated by Birbrair and Gabrielov in \cite{BG:2019}:
 \begin{conjecture}\label{conj:BG}
 Let $(X_0,0)$ and $(X_1,0)$ be two LNE semialgebraic surface germs which are ambient topologically equivalent and outer bi-Lipschitz equivalent. Then $X_0$ and $X_1$ are ambient bi-Lipschitz equivalent.
 \end{conjecture}

 In this paper, we obtain also a counterexample to Conjecture \ref{conj:BG} even for germs of LNE surfaces with isolated singularities. Indeed, we obtain even more as stated in the next result:
\begin{customthm}{\ref*{Cor:conjecture-false}}
For any non-trivial knot $K$ there exists two LNE surfaces $Y_K, \tilde Y_K \in \R^4$, with isolated singularity at $0$, such that
\begin{enumerate}
	\item $Y_K$ and $\tilde Y_K$ are outer bi-Lipschitz equivalent;
	\item $Y_K$ and $\tilde Y_K$ are ambient topologically equivalent;
	\item $Y_K$ and $\tilde Y_K$ are not ambient bi-Lipschitz equivalent.
\end{enumerate}
\end{customthm}


\section{Preliminaries}
Throughout the paper \emph{definable} means definable in a polynomially bounded o-minimal structure over $\mathbb{R}$ with the field of exponents $\mathbb{F}$.
All sets, functions and maps in this text are assumed to be definable. Unless the contrary is explicitly stated, we consider only germs at the origin of all sets and maps.

\begin{definition}\label{Def:tangentcone}\normalfont
	Given $X \subset \R^n$, the \emph{tangent cone} $C_0 X$ of $X$ at $0$ is the set defined as the real cone (consisting of half-lines), with vertex at $0$, over the following set:
	$$\mathbb{S}_0X:=\lim_{\varepsilon\to 0^+}\left(\,\frac{1}{\varepsilon}\big(X \cap\{\|x\|=\varepsilon\}\big)\right).$$
	Here we consider the Hausdorff limit (the definability ensures its existence). 
\end{definition}
Thus, if $\mathbb{S}_0X$ is defined as above, $C_0 X=\{tv;v\in \mathbb{S}_0X$ and $t\geq 0\}$.
\begin{definition}\label{Def: Link at the origin}\normalfont
	Given $X \subset \R^n$, the \emph{link of $X$} (at the origin) is the topological equivalence class of the sets $X\cap \mathbb{S}^{n-1}_{\varepsilon}$ for small enough $\varepsilon>0$. We denote the link of $X$ as $Link(X)$. The \emph{tangent link of $X$} is $Link(C_0X)$.
\end{definition}

\begin{remark}\label{Rem: well-defined} \normalfont
	The notion of link used in this paper is different from the notion of link in Knot Theory.
	Thus, if $(X,0)\subset \R^4$ is a surface germ with an isolated singularity, then each connected component of $Link(X)$ is a knot in $\mathbb{S}^3$.
\end{remark}

\begin{definition}\label{Def:bi-Lipschitz map}\normalfont
	Given two metric spaces $(X_1, d_1)$ and $(X_2, d_2)$, we say that a bijection $\varphi: X_1 \to X_2$ is \emph{bi-Lipschitz for the metrics $d_1$ and $d_2$} (or bi-Lipschitz, for short) if there is a real number $C \ge 1$ such that
	\begin{equation*}
		\frac{1}{C} \cdot d_1(p, q) \leq d_2(\varphi(p), \varphi(q)) \leq C \cdot d_1(p, q), \quad \forall p, q \in X_1.
	\end{equation*}
	For each $C \ge 1$ that satisfies such condition, we say that $\varphi$ is $C$-bi-Lipschitz. Furthermore, we say that the metric spaces $(X_1, d_1)$ and $(X_2, d_2)$ are \emph{bi-Lipschitz equivalent} if there is a bi-Lipschitz map $\varphi: X_1 \to X_2$.
\end{definition}

\begin{definition}\label{Def:inner-outer-metric}\normalfont
	Given $X \subset \mathbb{R}^{n}$, we define the following two metrics in $X$:
	\begin{itemize}
		\item \emph{the outer metric of $X$}: we define the outer metric of $X$ as the Euclidean distance $d_{out}: X \times X \to \mathbb{R}_{+}$, given by $d_{out}(x, y) = \|x - y\|$, for all $x, y \in X$;
		\item \emph{the inner metric of $X$}: if $X$ is path-connected, we define the inner metric of $X$ as the distance $d_X=d_{inn}: X \times X \to \mathbb{R}_{+}$, given by $d_{inn}(x, y) = \inf{\{ \ell(\alpha) \}}$, for all $x, y \in X$. The infimum is taken over all rectifiable paths $\alpha \subset X$ from $x$ to $y$, and $\ell(\alpha)$ is the length of $\alpha$.
	\end{itemize}
\end{definition}

\begin{remark}\label{Rem: inner-outer-metric-germ}
	Note that the germ of a closed definable set $X\subset \mathbb{R}^{n}$ is always connected and then it is path-connected. Therefore, $d_{out}(x,y) \le d_{inn}(x,y)<\infty, \forall x,y\in X$.
\end{remark}

\begin{remark}\label{Rem: pancake metric}
	We do not know whether the inner metric is definable (it is claimed by the Hardt conjecture, still open), but one can consider an equivalent definable metric as in \cite{KurdykaOrro97} (see also the \textbf{pancake metric} defined in \cite{LBirbMosto2000NormalEmbedding}). 
\end{remark}

\begin{definition}\label{Def:outer, inner, ambient bi-Lipschitz map}\normalfont
	Given two sets $X_1 \subset \mathbb{R}^{m}$, $X_2 \subset \mathbb{R}^{n}$, we consider three kinds of bi-Lipschitz equivalence by asking for the existence of a homeomoprhism from $X_1$ to $X_2$ that is additionally an
	\begin{itemize}
		\item \emph{outer bi-Lipschitz map}, i.e. 
  a map that is bi-Lipschitz for the outer metrics of $X_1$ and $X_2$;
		\item \emph{inner bi-Lipschitz map} i.e. 
  (assuming $X_1$ and $X_2$ are path-connected) a map that is bi-Lipschitz for the inner metrics of $X_1$ and $X_2$.
		\item \emph{ambient bi-Lipschitz map} i.e. 
  a bi-Lipschitz map $\varphi : \mathbb{R}^{m} \to \mathbb{R}^{n}$ with respect to the outer metric, such that $\varphi (X_1) = X_2$. Notice that if such an ambient bi-Lipschitz map exists, then $m=n$.
	\end{itemize}
	We also say that $X_1$ and $X_2$ are \emph{outer bi-Lipschitz equivalent} (resp. inner, ambient) if there is an outer bi-Lipschitz (resp. inner, ambient) map between $X_1$ and $X_2$.
\end{definition}

\begin{remark}\label{Rem: amb-out-inn}\normalfont
	If $X_1$ and $X_2$ are ambient bi-Lipschitz equivalent, then $X_1$ and $X_2$ are outer bi-Lipschitz equivalent, and if $X_1$ and $X_2$ are outer bi-Lipschitz equivalent, then $X_1$ and $X_2$ are inner bi-Lipschitz equivalent. However, the converses are not generally true (for counterexamples, see \cite{withMisha}).
\end{remark}

\begin{definition}\label{Def:LNE}\normalfont
	Given a path-connected set $X \subset \mathbb{R}^{n}$, we say that $X$ is \emph{Lipschitz normally embedded} (or LNE, for short) if the outer metric and inner metric are bi-Lipschitz equivalent i.e. there is a constant $C \ge 1$ such that
	\begin{equation*}
		d_X(x,y) \leq C \cdot \|x-y\|; \quad \forall \; x,y \in X.
	\end{equation*}
	In this case, we say that $X$ is $C$-LNE. Given $p\in X$, we say that $X$ is \emph{Lipschitz normally embedded at $p$} (or LNE at $p$) if there exists a neighborhood $U$ of $p$ such that $X \cap U$ is LNE, or equivalently, the germ $(X,p)$ is LNE. If $X\cap U$ is $C$-LNE for some $C\ge 1$, we say that $(X,p)$ is $C$-LNE.
\end{definition}

\begin{theorem}[Mendes-Sampaio Criterion \cite{mendes-sampaio}]\label{Teo:mendes-sampaio}
	Let $X \subset \mathbb{R}^{n}$ be a closed subanalytic set, such that $0 \in X$ and the link of $X$ is connected. Then, $X$ is LNE at 0 if and only if there is a uniform constant $C\ge1$ such that $X_t$ is $C$-LNE, for all $t>0$ small enough.
\end{theorem}
The theorem above still holds true when $X$ is a definable set in a polynomially bounded o-minimal structure on $\R$ 
(see \cite{Nhan:2021} or \cite{Sampaio:2023}).

\begin{definition}\label{Def:arc}\normalfont
	An \emph{arc in $\mathbb{R}^{n}$} with initial point $p$ is a germ at the origin of a continuous, non-constant definable map $\gamma : [0, t_0) \to \mathbb{R}^{n}$, for some $t_0 >0$, such that $\gamma (0) = p$. Every arc with an initial point at the origin will be simply called an arc. Given a germ at the origin of a set $X$, the set of all arcs $\gamma \subset X$ is called the \emph{Valette link of $X$} and is denoted by $V(X)$ (see \cite{Valette-Link}).
\end{definition}

\begin{remark}\normalfont
	Usually, we identify an arc $\gamma$ at $p\in X$ with its pointwise image $\gamma(t) \in \mathbb{R}^{n}$ obtained by intersecting $\gamma$ with a sphere centred at $p$ with radius $t$, for $t$ sufficiently small. This intersection is unique, by the the Local Conical Structure Theorem. This means also that \emph{we can parametrize all the arcs by the Euclidean distance to the point} $p$: $||\gamma(t)-p||=t$.
\end{remark}

\begin{definition}\label{Def:tord}\normalfont
	Given a set $X$ and two arcs $\gamma_1, \gamma_2 \in V(X)$, we define the \emph{order of tangency of $\gamma_1, \gamma_2$ in the outer metric}, denoted as $tord(\gamma_1, \gamma_2)$, as the exponent $\beta \in \mathbb{F}$ such that there exists a constant $c>0$ satisfying
	\begin{equation*}
		\| \gamma_1 (t) - \gamma_2 (t) \| = ct^{\beta}+o(t^\beta)
	\end{equation*}
	Notice that, if $\gamma_1 \neq \gamma_2$ (as germs at the origin), by the Newton-Puiseux theorem, such constants $c$ and $\beta=\mathrm{ord}_0||\gamma_1-\gamma_2||$ exist. We also define $tord(\gamma,\gamma) :=\infty$ for every arc $\gamma \in V(X)$. We define the \emph{order of tangency of $\gamma_1, \gamma_2$ in the inner metric} similarly and denote it as $tord_{inn}(\gamma_1, \gamma_2):=\mathrm{ord}_0d_{inn}(\gamma_1,\gamma_2)$ (the definition is well-posed, see Remark \ref{Rem: pancake metric}).
\end{definition}

\begin{remark}\label{Rem:tord}\normalfont
	Both orders of tangency are elements of $\F$ satisfying $1 \le tord_{inn}(\gamma_1,\gamma_2) \le tord(\gamma_1, \gamma_2)$ for all $\gamma_1, \gamma_2 \in V(X)$. Moreover, $X$ is LNE at 0 if, and only if, $tord_{inn}(\gamma_1,\gamma_2) \ge tord(\gamma_1, \gamma_2)$, for all $\gamma_1, \gamma_2 \in V(X)$ (see \cite{withRodrigo}).
\end{remark}

\begin{definition}\label{Def:Holder-triangle}\normalfont
	Given $\alpha \in\mathbb{F}_{\ge 1}$, we define the \emph{$\alpha$-standard H\"older triangle} as the set:
	$$T_{\alpha} = \{ (x,y) \in \mathbb{R}^{2} \mid 0\le x \le 1 ; 0\le y \le x^{\alpha} \}$$
The curves $l_0 := \{ (x,0) \in \mathbb{R}^{2} : 0\le x \le 1\}$ and $l_1 := \{ (x,x^{\alpha}) \in \mathbb{R}^{2} : 0\le x \le 1 \}$ are called the \emph{boundary arcs of $T_{\alpha}$}. We say that a set $X\subset \mathbb{R}^{n}$ is an \emph{$\alpha$-H\"older triangle with the main vertex at $a \in X$} if there exists an \emph{inner} bi-Lipschitz map $\varphi : (T_\alpha,0) \to (X,a)$. Here, $T_\alpha, X$ are seen as metric spaces with the inner metric. The sets $\gamma_0=\varphi (l_0)$, $\gamma_1=\varphi (l_1)$ are defined to be the \emph{boundary arcs of $X$} and we also denote the H\"older triangle $X$ as $T(\gamma_0,\gamma_1)$. Note that $\alpha$ is uniquely determined, i.e. if $X$ is also inner bi-Lipschitz to an $\alpha'$-standard H\"older triangle, then $\alpha'=\alpha$.
\end{definition}

Note that a H\"older triangle need not be LNE whereas the standard one is so. A similar observation is valid for the second type of basic `building block' which is the horn defined hereafter.

\begin{definition}\label{Def:horn}\normalfont
	Given $\beta \in \F_{\ge 1}$, the \emph{$\beta$-standard horn} is the set
	$$H_{\beta} = \{ (x,y,t) \in \mathbb{R}^{3} \mid 0\le t \le 1 \; ; \; x^2 + y^2 = t^{2\beta} \}.$$
	
	We say that a set $X\subset \mathbb{R}^{n}$ is a \emph{$\beta$-horn with the main vertex at $a \in X$} if there exists an \emph{inner} bi-Lipschitz map $\varphi : (H_\beta,0) \to (X,a)$. Here, $H_\beta$ and $X$ are seen as metric spaces with the inner metric. Again, $\beta$ is uniquely determined for $X$.
\end{definition}
\begin{remark}\label{Rem:horn}\normalfont
A useful observation is that whenever we choose a $\beta$-H\"older triangle $T\subset X$ inside a $\beta$-horn $X$ (with the same vertex), we get another $\beta$-H\"older triangle $\overline{X\setminus T}$ (see \cite{birbrair99}).
\end{remark}

\begin{definition}\label{Def:amb-trivial}\normalfont
	An LNE $\alpha$-H\"older triangle $T\subset \R^n$ is called \emph{Ambient Lipschitz trivial}  if it is ambient bi-Lipschitz equivalent to a standard $\alpha$-H\"older triangle in a two dimensional subspace of $\R^n$.
	
	An LNE $\beta$-horn $H\subset \R^n$ is called \emph{Ambient Lipschitz trivial} if it is ambient bi-Lipschitz equivalent to a standard $\beta$-horn triangle in a three dimensional subspace of $\R^n$.
\end{definition}

\section{Microknots and Hornifications}

\begin{definition}\label{Def:microknot}\normalfont
Let $\beta \in \F_{\ge 1}$, $K$ be a knot and $X\subset \R^4$ be an LNE surface germ. We say that $X$ contains an \emph{$\beta$-microknot}, corresponding to the knot $K$, if the following conditions holds:
\begin{enumerate}
	\item there are arcs $\gamma_1, \gamma_2 \in V(X)$ and a $\beta$-H\"older triangle $T \subset X$, with $\gamma_1$, $\gamma_2$ as its boundary arcs;
	\item there is an ambient trivial 1-H\"older triangle ${T'} \subset \mathbb{R}^4$, with $\gamma_1$, $\gamma_2$ as its boundary arcs, and $Link(T\cup T')$ is isotopic to $K$.
\end{enumerate}
\end{definition}

\begin{figure}
    \centering
     \includegraphics[width=1\textwidth]{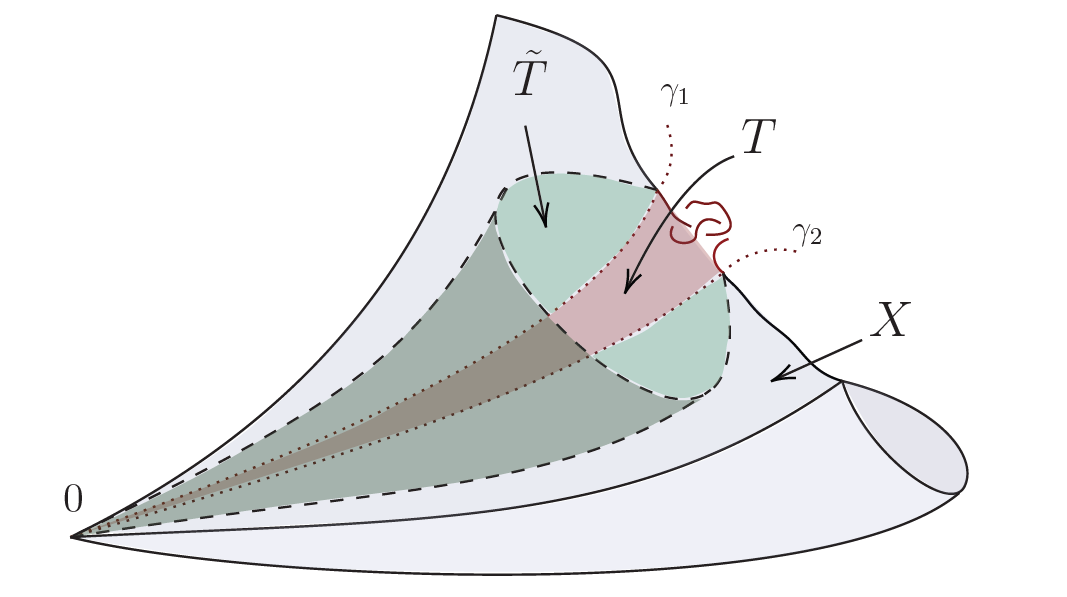}
    \caption{Figure 1: Surface $\displaystyle X$ containing a $\displaystyle \beta -$microknot in $\displaystyle T$.}
\end{figure}

Before the proof we shall introduce the construction of the so-called \emph{hornification}.

\begin{definition}\label{Def:horn-neighboorhood}\normalfont
	Let $X \subset \R^n$ be a germ at zero, $\eta \in (0,+\infty)$ and $\beta \in \F_{\ge 1}$. The \emph{$\beta$-horn neighborhood of amplitude $\eta$ of $X$} is defined as the set 
	\begin{equation*}
		\mathcal{H}_{\beta,\eta}(X):=\bigcup_{x\in X}\{z\in \R^n;\|z-x\|<\eta \|x\|^\beta\}.
	\end{equation*}
	If $\gamma\in V(\R^n)$ is an arc, we define the \emph{quasi-polar coordinates} of a point $x\in\mathcal{H}_{\beta,\eta}(\gamma)\setminus\gamma$ as $x=(t, \rho, v)$, where  $$\begin{cases} t=t(x)=\|x\|,\\ \rho=\rho(x)=\|x-\gamma(||x||)\|,\\ v=v(x) \in \mathbb{S}^{n-1}\end{cases}$$ with $v$ the tangent vector at $\gamma(||x||)$ of the geodesic in $\mathbb{S}_{||x||}^{n-1}$ connecting $x$ and $\gamma(||x||)$.
\end{definition}

\begin{definition}\label{Def: hornification}\normalfont
	Let $\eta \in (0,+\infty)$, $\beta \in \F_{\ge 1}$ and $K\subset\mathbb{S}^3$ be a knot. By the Akbulut-King Theorem (see \cite{Akbulut-King}), there is a smooth algebraic realisation of $K$. Let $\delta_K: [0,2\pi]\to \mathbb{S}^3$ be such a map ($\delta_K(0)=\delta_K(2\pi)$) and $\tilde K \subset \mathbb{S}^3$ be the image of $\delta_K$. Let also $\ell \subset \R^4$ be a straight line through the origin chosen so that $\ell \cap \tilde K=\emptyset$ and $\tilde K \subset \mathcal{H}_{\beta,\eta}(\ell) \cap \mathbb{S}^3$. Consider the following $\R^*_+$-action on $H_{\beta,\eta}(\ell)\setminus\ell$:
	\begin{equation*}
		\Theta_{\beta}(t,x)=(t,t^{\beta}\rho(x),v(x)),\quad \textrm{for}\> t>0, x\in H_{\beta,\eta}(\ell)\setminus\ell.
	\end{equation*} 
	The \emph{$\beta$-hornification of $\tilde K$ along $\gamma$} is the image germ $${X}_{\beta, K}:=\Theta_{\beta}(\R^*_+\times\tilde K)\cup\{0\}.$$
\end{definition}

\begin{remark}\label{Rem:action} \normalfont
	For each $\eta>0$ small enough, we can suppose, by performing a homothety in $\mathbb{S}^3$, if necessary, that the $\beta$-hornification of $\tilde K$ stays in $\mathcal{H}_{b,\eta}(\ell)$. Moreover, the orbits $\gamma_{\theta}(t)=\Theta_{\beta}(t,\delta_K(\theta))$ of the action make a foliation on ${X}_{\beta, K}$. Notice also that, for each $\theta \in [0,2\pi]$, the two-dimensional plane $P_\theta$ containing  $\ell$ and $\gamma_{\theta}(t)$ does not depend on $t$ and contains the arc $\gamma_{\theta}$ (extended to the origin by setting $\gamma_{\theta}(0)=0$).
\end{remark}

\begin{definition}\label{Def: generic}\normalfont
	Given $\theta \in [0,2\pi]$, we say that $\gamma_{\theta}$ is \emph{simple} if the geodesic $r_{\theta}:[0,2\pi] \to \mathbb{S}^3$, connecting $\ell(1)=r_\theta(0)=r_\theta(2\pi)$ and $\gamma_{\theta}(1)=\delta_K(\theta)$ satisfies the following conditions:
	\begin{enumerate}
	\item $r_{\theta}$ is transversal to $\tilde K$;
	\item if $0=\tau_0<\tau_1<\cdots<\tau_n=2\pi$ are real numbers such that $r_{\theta}\cap \tilde K=\{r_{\theta}(\tau_1),r_{\theta}(\tau_2),\cdots,r_{\theta}(\tau_{n-1})\}$ and $r_{\theta}(\tau_k)=\gamma_{\theta}(1)$, then $r_{\theta}((\tau_k,\tau_{k+1}))\cap \partial_{\mathbb{S}^3}(\mathcal{H}_{\beta,\eta}(\ell) \cap \mathbb{S}^3)=\{r_{\theta}(u_{\theta}),r_{\theta}(v_{\theta})\}$, where $0<u_{\theta}<v_{\theta}<2\pi$.
	\end{enumerate}	
\end{definition}

\begin{remark}\label{Rem:simple}\normalfont
	Since $\ell(1) \notin \tilde K$, the set of $\theta \in [0,2\pi]$ such that $\gamma_{\theta}$ is simple has a nonempty interior. Indeed, take $\theta \in [0,2\pi]$ such that the distance, in $\mathbb{S}^3$, between $\ell(1)$ and $\gamma_{\theta}(1)$ is maximal. Obviously, a small shift in $\theta$ will ensure that $r_\theta$ is transversal to $\tilde K$ (this is a compact set) and this will remain true for all $\theta'$ close enough to $\theta$, transversality being an open condition.  
\end{remark}

Now we show some properties of the hornifications. In what follows, consider the notations as in Definitions \ref{Def: hornification}, \ref{Def: generic} and Remark \ref{Rem:action}.

\begin{proposition}\label{Prop:hornification-LNE} 
	The $\beta$-hornification $X_{\beta, K}$ of $\tilde K$ is LNE at 0.
\end{proposition}

\begin{proof}
	Since the group $\R_{+}^{*}$ acts on $\tilde K$ by homotheties with ratio $t^{\beta}$ and centre $\ell(t)$, and $\tilde X_K$ is smooth, the supremum of the quotient $d_{inn}(x_1,x_2)/\|x_1-x_2\|$, restricted to $X_{\beta, K} \cap \mathbb{S}^3_{t}$, is a positive number and does not depend on $t$ (note that $\tilde K$ is LNE it being smooth). By Theorem \ref{Teo:mendes-sampaio}, the set $X_{\beta, K}$ is LNE.
\end{proof}

\begin{remark}\label{rem: HT}\normalfont
    In connection with Remark \ref{Rem:horn} it should be observed that the $\beta$-hornification $X_{\beta,K}$ is a $\beta$-horn. It follows from the construction that given any two distinct $\gamma_1,\gamma_2\in V(X_{\beta, K})$, we get an $\alpha$-H\"older triangle $T(\gamma_1,\gamma_2)=\Theta_\beta((0,1]\times K_{\gamma_1,\gamma_2})\cup\{0\}$ where $K_{\gamma_1,\gamma_2}$ is the part of $\tilde K$ lying between the two points $\gamma_i(1)=\delta_K(\theta_i)$, $i=1,2$, i.e. $K_{\gamma_1,\gamma_2}=\delta_K([\theta_1,\theta_2])\subset \tilde{K}$ (assuming $\theta_1<\theta_2$). Moreover, $\alpha=tord(\gamma_1,\gamma_2)\geq \beta$. In particular, if we choose $\gamma_i=\gamma_{\theta_i}$ for some $\theta_1<\theta_2$ in $[0,2\pi)$, then we will obtain a $\beta$-H\"older triangle, since both curves lie in the respective planes $P_{\theta_i}$ with $P_{\theta_1}\cap P_{\theta_2}=\ell$ and $tord_0(\gamma_{\theta_i},\ell)=\beta$ by construction. 
\end{remark}

\section{Universality Theorem for LNE H\"older triangles}
\begin{theorem}\label{main} For any knot $K$ and $\beta \in \mathbb{F}_{> 1}$, there exists an LNE H\"older triangle $T_{\beta, K} \subset \R^4$, containing a $\beta$-microknot, isotopic to $K$. Moreover, two triangles $T_{\beta, K_1}$ and $T_{\beta, K_2}$ of this kind are ambient bi-Lipschitz equivalent only if $K_1$ and $K_2$ are isotopic.
\end{theorem}

\begin{proof}
 By Remark \ref{Rem:simple}, there are $0<\theta_1<\theta_2<2\pi$ such that $\gamma_{\theta}$ is simple, for every $\theta \in [\theta_1, \theta_2]$. Fix such $\theta$ and let $Q_\theta$ be the plane spanned by the tangent vector to $\tilde K$ at $\delta_K(\theta)=\gamma_\theta(1)$ and the vector spanning $\ell$. Consider the orthogonal projection $\phi: X_{\beta, K} \to Q_{\theta}$. Let $T=T(\gamma_{\theta_1},\gamma_{\theta_2})\subset X_{\beta, K}$ be a H\"older triangle, such that  $\gamma_{\theta} \in V(T)$ (cf. Remark \ref{rem: HT}). By choosing $\theta_2$ close enough to $\theta_1$, we can suppose that $\phi|_{T}$ is a bi-Lipschitz map.

\begin{afirm}\label{small-triangle} $T$ is an ambient Lipschitz trivial H\"older triangle.
\end{afirm}

\begin{proof}
Let $\tilde T \subset Q_{\theta}$ be the image of $T$  under the map $\phi$. Notice that $T$ is a graph of a bi-Lipschitz map $\tilde \phi:\tilde T \to \R^2$. By the Kirszbraun Theorem, there is a definable Lipschitz extension $\Phi$ of $\tilde \phi$, from $\tilde T$ to the two-dimensional plane $Q_{\theta}$. By rotating axes, if necessary, let us define the coordinates $(u_1,u_2,v_1,v_2)$ on $\R^4$, such that $(u_1,u_2)$ are the coordinates in $Q_{\theta}$ and $(v_1,v_2)$ are the coordinates in its orthogonal complement. Now define $\Psi : \R^4 \to \R^4$ as: 
$$ \Psi(u_1,u_2,v_1,v_2)= (u_1,u_2, (v_1,v_2)-\Phi(u_1,u_2)) \, ; \, \forall \, (u_1,u_2,v_1,v_2) \in \mathbb{R}^4.$$ Since $\Phi$ is Lipschitz and  definable, and $\Psi$ is a bijection with inverse $$\Psi^{-1}(u_1,u_2,v_1,v_2)=(u_1,u_2,(v_1,v_2)+\Phi(u_1,u_2)),$$ it follows that $\Psi$ is a definable  bi-Lipschitz map. Notice also that $\Psi(T)=\tilde T$.  This proves the affirmation.
\end{proof}

Let $x_1, x_2 \in [0,2\pi]$ such that $r_{\theta_1}(x_1),r_{\theta_2}(x_2) \in \mathbb{S}^3 \setminus \mathcal{H}_{\beta, \eta}(\ell)$ and the geodesic segment $\tau \subset \mathbb{S}^3$ connecting $r_{\theta_1}(x_1)$ and $r_{\theta_2}(x_2)$ is transversal to $r_{\theta_1}$ and $r_{\theta_2}$. For $i=1,2$, let $\ell_i$ be the straight line from $0$ to $r_{\theta_i}(x_i)$ and let $T_{\theta_i}=T(\gamma_{\theta_i},\ell_i) \subset P_{\theta_i}$ (cf. Remark \ref{Rem:action}). Notice that $T_{\theta_1}$ and $T_{\theta_2}$ are 1-H\"older triangles. Finally, define:
 $$T_{\beta, K}:= (X_{\beta, K} \backslash T)\cup (T_{\theta_1}\cup T_{\theta_2})
  \> \textrm{and} \> T':=\mathrm{Tr}_0(\tau)$$
 with $\mathrm{Tr}_0(\tau)=\{tv\mid v\in \tau, 0\leq t\leq 1\}$.

\begin{figure}
    \centering
     \includegraphics[width=1\textwidth]{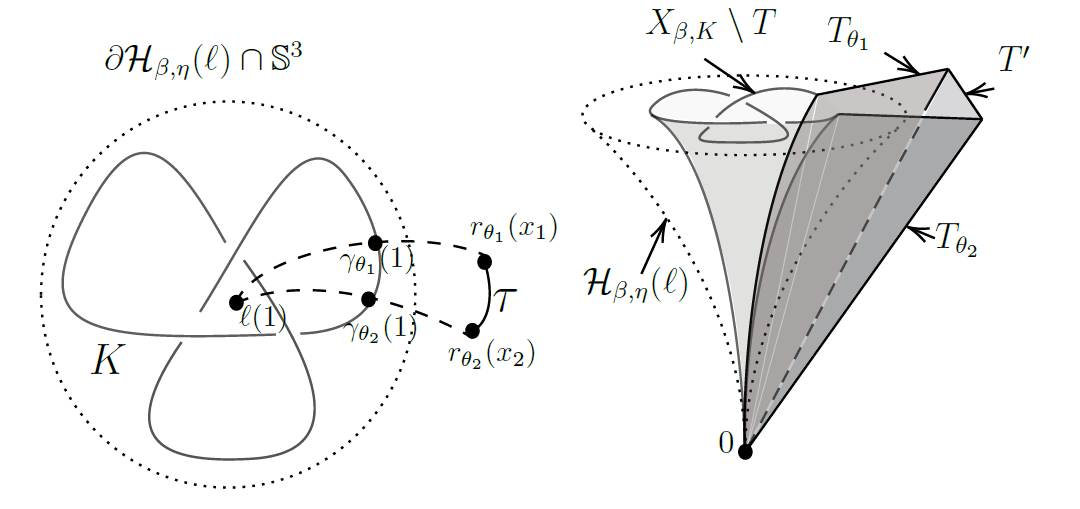}
    \caption{Figure 2: Construction of the sets $\displaystyle T_{\beta, K}$ and $\displaystyle {T'}$.}
\end{figure}

We will prove first that $T_{\beta, K}$ is LNE. The set $X_{\beta, K} \backslash T$ is LNE because $X_{\beta, K}$ is LNE by Proposition \ref{Prop:hornification-LNE} and since $\theta_1<\theta_2$, $T=T(\gamma_{\theta_1},\gamma_{\theta_2})$ is an LNE $\beta$-H\"older triangle due to the action $\Theta_{\beta}$ (Remark \ref{rem: HT}). By Remark \ref{Rem:horn} and  $\overline{X_{\beta, K} \backslash T}$ is a $\beta$-H\"older triangle as well and it is LNE automatically. 

Notice also that $T_{\theta_1}$ and $T_{\theta_2}$ are LNE simply because they are clearly plane 1-H\"older triangles. Suppose now that $\gamma_1\in V(X_{\beta, K} \backslash T)$ and $\gamma_2\in V(T_{\theta_1})$. Since the angle between the tangent spaces along $\gamma_{\theta_1}$ to the set $T_{\beta, K}$ and the plane $P_{\theta_1}$ containing $T_{\theta_1}$ is bounded away from zero ($\gamma_{\theta_1}$ being simple), we obtain, by the LNE conditions satisfied by both $X_{\beta, K} \backslash T$ and $T_{\theta_1}$, that 
$$tord(\gamma_1,\gamma_2)=\min\{tord_{inn}(\gamma_1,\gamma_{\theta_1}), tord_{inn}(\gamma_2,\gamma_{\theta_1})\}=tord_{inn}( \gamma_1,\gamma_2).$$ 
Therefore, $(X_{\beta, K} \backslash T)\cup T_{\theta_1}$ is LNE, and analogously $(X_{\beta, K} \backslash T)\cup T_{\theta_2}$ is LNE. Thus $T_{\beta, K}$ is LNE. In a similar way, $T_{\beta,K}  \cup T'$ is LNE, too.

Let us show that $T_{\beta, K}$ has a microknot. Notice that $T'=\mathrm{Tr}_0\tau$ is ambient Lipschitz trivial, since it is a plane standard 1-H\"older triangle. Since $T_{\beta, K}$ is a $\beta$-H\"older triangle, $T_{\beta, K}, T'$ have $\ell_1, \ell_2$ as boundary arcs and $Link(T_{\beta, K} \cup T')$ is isotopic to $K$ by construction (the arcs $\gamma_{\theta_1},\gamma_{\theta_2}$ beign simple), the result follows.

To end the proof, suppose that $T_{\beta,K_1}$ and $T_{\beta,K_2}$ are ambient bi-Lipschitz equivalent. Let $\varphi:\R^4 \to \R^4$ be an ambient bi-Lipschitz map, preserving the orientation of the knots and such that $\varphi(T_{\beta,K_1})=T_{\beta,K_2}$. Then the unions $T_{\beta,K_1} \cup T'_{K_1}$ and $T_{\beta,K_2} \cup \varphi(\tilde T'_{K_1})$ are ambient topologically equivalent (here $T'_{K_1}$ is the triangle $T'$ constructed for $T_{\beta,K_1}$). That is why $X_{\beta,K_1} \simeq T_{\beta,K_1} \cup T'_{K_1}$ and $X_{\beta,K_2} \simeq T_{\beta,K_2} \cup  \varphi(T'_{K_1})$ must also be ambient topologically equivalent. Hence, ${K_1}$ and ${K_2}$ are isotopic.
\end{proof}

\section{Counterexample to the conjecture of Birbrair and Gabrielov}

\begin{theorem}\label{Cor:conjecture-false}
For any non-trivial knot $K$ there exists two LNE surfaces $Y_K, \tilde Y_K \subset \R^4$, with isolated singularity at $0$, such that

\begin{enumerate}
	\item $Y_K$ and $\tilde Y_K$ are outer bi-Lipschitz equivalent;
	\item $Y_K$ and $\tilde Y_K$ are ambient topologically equivalent;
	\item $Y_K$ and $\tilde Y_K$ are not ambient bi-Lipschitz equivalent.
\end{enumerate}
\end{theorem}

\begin{proof}
For a fixed $\beta>1$, let $Y_K=T_{\beta, K} \cup T'$ with $T'$ constructed as in the proof of Theorem \ref{main}, and let $\tilde Y_K =\mathrm{Tr}_0\tilde K=\Theta_1((0,1]\times \tilde K)\cup\{0\}$. Since $Y_K$ is LNE, $Y_K$ contains 1-H\"older triangles $T'$ and $T_{\theta_i}$, $i=1,2$, and $Y_K$ has a link isotopic to $K$, $Y_K$ is outer bi-Lipschitz equivalent to a 1-horn. $\tilde Y_K$ is the cone over a knot isotopic to $K$, and $\tilde Y_K$ is LNE by Proposition \ref{Prop:hornification-LNE}. Therefore $Y_K$ and $\tilde Y_K$ are outer bi-Lipschitz and ambient topologically equivalent, and  isolated singularity at $0$ (formally we should replace $Y_K$ with the 1-horn it is outer bi-Lipschitz equivalent to).

Suppose now that $Y_K$ and $\tilde Y_K$ are ambient bi-Lipschitz equivalent. Then, $C_0 Y_K$ and $C_0 \tilde Y_K$ are ambient bi-Lipschitz equivalent by \cite{Sampaio-tangent}. Since $\beta>1$, the tangent cone of $ X_{\beta,K}$ is $\ell$, and thus $C_0 Y_K$ is $C_0(T_{\theta_1}\cup T_{\theta_2} \cup T')$, which is ambient Lipschitz trivial. On the other hand, $C_0 \tilde Y_K=\tilde Y_K$ (as germs) has the link isotopic to $K$, so it  cannot be trivial. Therefore, $C_0 Y_K$ and $C_0 \tilde Y_K$ are not ambient topologically equivalent, a contradiction.
\end{proof}

\end{document}